\documentclass[runningheads]{llncs}
\usepackage[T1]{fontenc}
\usepackage{xcolor}
\usepackage{amsmath,amssymb,geometry}
\usepackage{graphicx}
\usepackage{hyperref}
\usepackage{subfig}
\usepackage{biblatex}
\begin{document}
\title{Neumann-Neumann Waveform Relaxation Method for Hyperbolic PDE with Time Delay}
%
%
\author{Deeksha Tomer\inst{1} }
\authorrunning{D Tomer}
%
\institute{Department of Mathematics, Indian Institute of Technology Bhubaneswar\\
\email{a21ma09002@iitbbs.ac.in}}
\maketitle              
\begin{abstract}
Hyperbolic PDEs with time delay are essential for modeling a wide range of real-world applications. Since closed-form solutions are often not attainable, developing accurate and efficient numerical methods becomes crucial. In this work a novel substructuring waveform relaxation method namely  Neumann-Neumann Waveform Relaxation (NNWR) has been implemented for numerically solving hyperbolic delay PDEs in the case of non-uniform subdomain partitioning. The convergence behavior is first analyzed using Fourier analysis to obtain an estimate, followed by a demonstration of finite-step convergence through Laplace transform techniques. Various test cases are presented as numerical illustrations. \footnote{Accepted in 12th International Conference on Mathematics and Computing, 2026 }


\end{abstract}

\section{Introduction}
Hyperbolic partial differential equations (PDEs) with time delay play a central role in modeling phenomena where wave propagation is influenced by past system states. These delays can originate from signal transmission times, memory effects, or feedback mechanisms, and they add a layer of complexity to the dynamics. This makes both analytical and numerical treatment significantly more challenging compared to classical hyperbolic PDEs. Time-delayed hyperbolic equations arise in numerous fields including electrodynamics, population dynamics, traffic flow modeling, and viscoelastic materials \cite{schiesser2019time, erneux2009applied}.
In engineering applications, such models are used to represent delayed control systems and wave propagation in media with memory, such as those seen in aerospace engineering and telecommunications \cite{erneux2009applied}. In neuroscience, delayed hyperbolic equations model the time it takes for signals to travel through neurons and synapses \cite{schiesser2019time}. 

Numerically solving such systems is non-trivial, especially when the spatial domain is large or heterogeneous. Parallel numerical methods for solving hyperbolic PDEs with time delays are understudied in the literature and remain largely unexplored. This motivates the use of domain decomposition methods, which allow for parallel computation and better scalability. 

We explore the NNWR method, which specifically applies Neumann transmission conditions at the interfaces between subdomains. It is especially effective for hyperbolic problems than Dirichlet-based methods, as it can be naturally extended to nonuniform multiple subdomains \cite{Gander2003,sana2023dirichlet,mandal2017neumann}, which are frequently encountered in practical applications. Moreover, waveform relaxation \cite{Gander2015} is highly amenable to parallelization, making it attractive for high-performance computing applications . In the NNWR framework, initial Dirichlet problems are solved independently in each subdomain, after which a Neumann-based correction is applied at the interface.
Consider the classical wave equation defined over the spatial domain \(\Omega = (0, D) \):
\[
\frac{\partial^2 w}{\partial t^2} = c^2 \frac{\partial^2 w}{\partial x^2}+p(x,t), \quad x \in \Omega, \quad t > 0,
\]
along with the initial conditions given as
\[
w(x, 0) = w_0(x), \quad \frac{\partial w}{\partial t}(x, 0) = v_0(x),
\]
and the boundary conditions specified as
\[
w(0, t) = q_0(t), \quad w(D, t) = q_D(t).
\]
We decompose the given domain \(\Omega=(0,D) \) into two disjoint subdomains $D_1 = (0, \Gamma)$ and $D_2 = (\Gamma, D),$
where \( \Gamma \in (0, D) \) is the interface point.
\subsection*{NNWR Iterative Procedure}
At each iteration \( k \), the NNWR method perform following steps:
\subsection*{Step I: Solve Dirichlet subproblems on the subdomains \( D_1 \) and \(D_2 \)}
\[ 
\left\{\begin{array}{rl}
\partial_{tt} w_1^k&=c^2\partial _{xx} w_1^k +p,\quad\text{in}\quad D_1\times[0,T],\\ 
w_1^k(x,0)&=w_0(x), \\
\partial_tw_1^k(x,0)&=v_0(x),\\
 w_1^k(0,t)&=q_0(t),\\ 
 w_1^k(\Gamma,t)&= l^{k-1}(t),\\
\end{array}\right.\ \  
\left\{\begin{array}{rl}
\partial_{tt} w_2^k&=c^2\partial _{xx} w_2^k+p,\quad\text{in}\quad D_2\times[0,T],\\
w_2^k(x,0)&=w_0(x), \\
\partial_tw_2^k(x,0)&=v_0(x),\\
w_2^k(\Gamma,t)&= l^{k-1}(t),\\
w_2^k(D,t)&=q_D(t).\\ 
\end{array}\right.
\]

\subsection*{Step II: Solve Neumann subproblems on the subdomains \( D_1 \) and \( D_2 \)}
\[ 
\left\{\begin{array}{rl}
\partial_{tt} \psi_1^k&=c^2\partial _{xx} \psi_1^k,\quad\text{in}\quad D_1\times[0,T],\\ 
 \psi_1^k(x,0)&=0, \\
 \partial_t\psi_1^k(x,0)&=0,\\
  \psi_1^k(0,t)&=0,\\
 \partial_n\psi_1^k(\Gamma,t)&= \partial_nw_1^k+\partial_nw_2^k,\\
\end{array}\right.\ \  
\left\{\begin{array}{rl}
\partial_{tt}\psi_2^k&=c^2\partial _{xx} \psi_2^k,\quad\text{in}\quad D_2\times[0,T],\\
\psi_2^k(x,0)&=0, \\
\partial_t\psi_2^k(x,0)&=0,\\
\partial_n \psi_2^k(\Gamma,t)&=\partial_nw_1^k+\partial_nw_2^k,\\
\psi_2^k(D,t)&= 0,
\end{array}\right.
\]
along with the update condition applied at the interface
\[
l^{k}(t)=l^{k-1}(t)-\theta\{ \psi_1^{k}\mid _{\Gamma\times(0, T)} +\psi_2^{k}\mid _{\Gamma\times(0, T)}\}.\]
where relaxation parameter denoted by \( \theta \in (0, 1] \) and $l^0$ is the initial approximation at the interface.
The model problem under consideration is a linear second order wave equation incorporating a constant time lag or delay $\tau$, as considered in  \cite{Rodriguez},
\begin{equation}
\label{eq_1}
v_{tt}=c^2v_{xx}+ \lambda v(x,t-\tau),\  t>\tau, x \  \text{in} \ \Omega\subset \mathbb{R}^d,
 \end{equation}
subject to  an initial condition
$$ v(x,t)=\rho (x,t),\  0\leq t\leq \tau, x \  \text{in} \ \Omega, $$
along with the Dirichlet conditions on boundary,
$$
 v(x,t)=\chi(t), t\geq 0, x\in\partial\Omega.
$$
 Where $c$ represents the speed of wave propagation and $\lambda$ is a freely chosen parameter. Well-posedness has been thoroughly addressed in the studies presented in \cite{liu2016well,fridman2009exponential}.
 

In this work, we implement the NNWR method for solving hyperbolic PDEs with time delay in non-uniform subdomains. We begin by establishing convergence estimates via Fourier analysis and then derive conditions for finite-step convergence using Laplace transform techniques, providing a rigorous theoretical foundation for the method. The proposed approach is validated through numerical experiments that highlight its efficiency and robustness under various domain decompositions and time window lengths.

\section{ Analysis of NNWR for Asymmetric Decomposition }

The NNWR method for two non-overlapping subdomains of unequal size, \(D_1 = (-a, 0)\) and \(D_2 = (0, b)\), with an initial value \(q^0(t)\) as a guess  along the interface boundary \(\Gamma\), is described as:\\
Dirichlet subproblems formulated on subdomains $D_j$, where $j=1,2$ for the error equations are:
\begin{equation}\label{eq_8}
\left\{\begin{array}{rl}
\partial_{tt} e_j^k-c ^2\partial _{xx}e_j-\lambda e_j(x, t-\tau )&=0, \ \ \ (x, t) \text{ in }  D_j\times (0, T), \\ 
 e_j^k(x, t)&=0, \ \ \ (x, t) \text{ in } D_j\times[-\tau,0], \\ 
 \partial_te_j^k(x, t)&=0, \ \ \ (x, t) \text{ in } D_j\times[-\tau,0], \\ 
e_j^k&=0, \ \ \  \textnormal{on }  \partial D_j\backslash\Gamma,  \\ 
e_j^k&=q^{k-1},\ \ \  \textnormal{on } \Gamma.
\end{array}\right.
\end{equation}
Next, the Neumann subproblems are solved within the subdomains $D_j$, $j=1,2$:
\begin{equation}\label{eq_9}
\left\{\begin{array}{rl}
\partial_{tt} \psi_j^k-c ^2\partial _{xx}\psi_j^k-\lambda \psi_j^k(x, t-\tau )&=0,\ \ \  (x, t)\text{ in } D_j\times (0, T),  \\ 
 \psi_j^k(x, t)&=0,\ \ \  (x, t)\text{ in } D_j\times[-\tau,0], \\ 
 \partial_t\psi_j^k(x, t)&=0,\ \ \  (x, t) \  \text{in} \ D_j\times[-\tau,0], \\
 \partial _n\psi_j^k&=\partial _ne_1^k+\partial_ne_2^k,\ \ \   \textnormal{on }   \Gamma,\\
\psi_j^k&=0,\ \ \   \textnormal{on } \partial D_j\backslash\Gamma.
\end{array}\right.
\end{equation}
 The condition used to update across the interface boundary is:
\begin{equation}\label{eq_10}
q^{k}(t)=q^{k-1}(t)-\theta\{ \psi_1^{k}\mid _{\Gamma\times(0, T]} +\psi_2^{k}\mid _{\Gamma\times(0, T]}\}.
\end{equation}
\subsection{Fourier-Based Convergence Estimate}
Before proving the theorem (\ref{fourier:thmnnwr}) first we consider following Lemmas.
\begin{lemma}\label{lem_4}
    The function $g(x) = \frac{\sinh^2((a-b)x)}{\sinh(2ax)\sinh(2bx)}$,$\forall x>0$ and $a,b>0$ is a monotonically decreasing function and lies in between $\left ( 0,\frac{(a-b)^2}{4ab} \right )$.
\end{lemma} 
\begin{proof}
Since $\sinh(x)$ is an increasing function, so evaluating $g(x)$ as $x\to \infty$ we get,
$$
\quad \frac{\sinh^2((a-b)x)}{\sinh(2ax)\sinh(2bx)} \to 0,
$$
because \( \sinh(2ax)\sinh(2bx) \) dominates \( \sinh^2((a-b)x) \) for \( a, b > 0 \).
Since \( g(x) \) is decreasing (i.e., $g'(x)<0$) and approaches \( 0 \) as \( x \to \infty \), it must be monotonically decreasing for all \( x > 0 \).
$$\text{As } x \to 0, \quad \frac{\sinh^2((a-b)x)}{\sinh(2ax)\sinh(2bx)} \to \frac{(a-b)^2}{4ab},$$
So $$\forall x> 0, f(x)\in\left ( 0,\frac{(a-b)^2}{4ab} \right ).$$
\end{proof}
\begin{lemma}\label{lem_5} 
Let \( \Re(u) \) and \( \Im(u) \) correspond to the real component and imaginary component of the complex number \( u \), respectively. Consider the function 
\[
f(u) = \frac{\sinh^2((a-b)u)}{\sinh(2au)\sinh(2bu)},
\]
where \( a, b > 0 \). The function \( f(u) \) satisfies the inequality 
\[
|f(u)| < \frac{(a-b)^2}{4ab}, \quad \forall u \in \mathbb{D},
\]
where 
\[
\mathbb{D} = \left\{ u \in \mathbb{C} \; \middle| \; \Re(u) > 0, \; \Im(u) > 0, \; \Im(u) < \Re(u) \right\}.
\]
\end{lemma}
\begin{proof}
    The function $f(u)$ is analytic in the whole domain $\mathbb{D}$, thus the absolute value of given function is,
    $$\left | f(u)\right |=\frac{\sinh^2\left ( \left ( a-b \right )\Re (u) \right )+\sin^2((a-b)\Im (u))}{\sqrt{\sinh^2(2a\Re (u))+\sin^2(2a\Im(u))}\sqrt{\sinh^2(2b\Re (u))+\sin^2(2b\Im(u))}}$$
    since,$\Im(u)<\Re(u)$ so   $ |f(u)|<\frac{2\sinh^2((a-b)\Re(u))}{\sinh(2a\Re(u))\sinh(2b\Re(u))}$,\\
    then by using Lemma \ref{lem_4} we get, $|f(u)|<2\frac{(a-b)^2}{4ab}$.\\
    Since we get maximum value at the boundary of the domain $\mathbb{D}$, we derive following inequality:
    $$\left | f(u)\right |=\max\left\{ \frac{\sinh^2((a-b)\Re(u))}{\sinh(2a\Re(u))\sinh(2b\Re(u))},\left | \frac{\sinh^2((a-b)(1+i)\Re(u))}{\sinh(2a(1+i)\Re(u))\sinh(2b(1+i)\Re(u))}\right | \right\}$$
    By theorem $15.1$ in \cite{complex}, we get $|f(u)|<\frac{(a-b)^2}{4ab} \forall u \in \mathbb{D}$.
\end{proof}
\begin{theorem}

 The NNWR algorithm as defined in equations (\ref{eq_8})-(\ref{eq_10}) for the hyperbolic partial differential equation (PDE) with time delay (\ref{eq_1}), achieves linear convergence for $t\in (0,\infty)$ when $\theta=1/4$ with the estimate, $$\left\| q^k\right\|_{L^2(\Gamma _T)}\leq \left ( \frac{(a-b)^2}{4ab} \right )^k\left\| q^0\right\|_{L^2(\Gamma_T)}.$$
 \label{fourier:thmnnwr}
\end{theorem}
\begin{proof}
 By defining
\[
\hat{e}_j^k(x, \omega) = \frac{1}{2\pi} \int_{-\infty}^\infty e_j^k(x, t) \, e^{-i\omega t} \, dt,
\]
as the Fourier transform of \( e_j^k(x, t) \) in time and assume that
\[
\xi^2 = \frac{\omega^2 + \lambda e^{-i\omega \tau}}{c^2},
\] which incorporates the effect of the time delay into the frequency-domain representation. Then the solution for Dirichlet problems after applying Fourier transform for subdomain $D_1$ and $D_2$ are,
    \begin{align*}
\hat e_1^k(x,\omega)&=\frac{\hat q^{k-1}(\omega)\sinh(\xi (a+x))}{\sinh(\xi a)},
\\
\hat e_2^k(x,\omega)&=-\frac{\hat q^{k-1}(\omega)\sinh(\xi (b-x))}{\sinh(\xi b)}.\\
\end{align*}
Now the solutions for Neumann subproblems for subdomain $D_1$ and $D_2$ are,
\begin{align*}
\hat \psi _1^k(x,\omega)&=\frac{\hat q^{k-1}(\omega)(\coth(\xi a)+\coth(\xi b))\sinh(\xi (a+x))}{\cosh(\xi a)},
\\
\hat \psi_2^k(x,\omega)&=-\frac{\hat q^{k-1}(\omega)(\coth(\xi a)+\coth(\xi b))\sinh(\xi (b-x))}{\cosh(\xi b)}.\\
\end{align*}
Now from the update condition we get,
$$\hat q^k=\hat q^{k-1}\left [ 1-\theta \left ( 2+\frac{\tanh(\xi a)}{\tanh(\xi b)}+\frac{\tanh(\xi b)}{\tanh(\xi a)}\right ) \right ],$$
for $\theta = 1/4$,
$$ \hat q^k=\left [ \frac{1}{2}-\frac{1}{4}\left ( \frac{\sinh(\xi a)\cosh(\xi b)}{\cosh(\xi a)\sinh(\xi b)} +\frac{\sinh(\xi b)\cosh(\xi a)}{\sinh(\xi a)\cosh(\xi b)} \right )
\right ]\hat q^{k-1},$$
which on simplification gives,
$$\hat q^k=-\frac{\sinh^2((a-b)\xi)}{\sinh(2\xi a)\sinh(2\xi b)}\hat q ^{k-1}.$$
By using Lemma (\ref{lem_5}) we get,
$$\hat q^k=-\frac{(a-b)^2}{4ab}\hat q^{k-1},$$
Finally, using Parseval-Plancherel identity gives, $\left\| q^k\right\|_{L^2(\Gamma _T)}\leq \left ( \frac{(a-b)^2}{4ab} \right )^k\left\| q^0\right\|_{L^2(\Gamma_T)}.$\\
 Thus the NNWR method converges when $\frac{(a-b)^2}{4ab}<1$.
 \end{proof}
 \textbf{Remark} The NNWR converges when $a/b$ belongs to $\left ( \frac{6-\sqrt{32}}{2},\frac{6+\sqrt{32}}{2} \right )$, which we obtain by solving $\frac{(a-b)^2}{4ab}<1\Rightarrow \left ( (a/b)^{2}-6\left ( a/b \right )+1 \right )< 0$.
\subsection{Laplace Transform Based Convergence}
\begin{theorem}
    The NNWR algorithm for the Hyperbolic PDE with constant time delay attains convergence for $\theta=1/4$ in at most $k+1$ iterations for asymmetric decomposition of a domain in two subdomains, if the length of the time window $T$ fulfills the condition $ T/k\leq 4 \min\left\{ a/c,b/c\right\}$, where $a, b$ are subdomain sizes and $c$ is the speed of wave.
 \end{theorem}

 \begin{proof}
     First we apply the Laplace transform and then by induction, the NNWR kernel becomes,
 $$
 q^k(s)=\left [ 1-\theta\left(2+\frac{\tanh\left(a\frac{\sqrt{s^2-\lambda e^{-\tau s}}}{c }\right)}{\tanh\left(b\frac{\sqrt{s^2-\lambda e^{-\tau s}}}{c}\right)}+\frac{\tanh\left(b\frac{\sqrt{s^2-\lambda e^{-\tau s}}}{c }\right)}{\tanh\left(a\frac{\sqrt{s^2-\lambda e^{-\tau s}}}{c }\right)}\right) \right ]^kq^0(s).$$\\
A detailed derivation of this kernel is in \cite{ima}.
  For asymmetric decomposition $a\neq b$, we have
 $$\hat q^k(s)= \left\{ \left ( 1-4\theta \right )-\theta\left ( K^a_b(s)+K^b_a(s) \right )\right\}^k\hat q^0(s), k=1,2\ldots, $$
where $K^a_b(s)= \left (\frac{\tanh\left(a\frac{\sqrt{s^2-\lambda e^{-\tau s}}}{c }\right)}{\tanh\left(b\frac{\sqrt{s^2-\lambda e^{-\tau s}}}{c}\right)}-1\right)$ and the update equation further reduces with $\theta =1/4$, 
 $$\hat q^k(s)= \left\{ -(1/4)\left ( K^a_b(s)+K^b_a(s) \right )\right\}^k\hat q^0(s).$$
Using exponential series expansion and assuming $X=\frac{\sqrt{s^2-\lambda e^{-\tau s}}}{c}$ we get,

{\small
\begin{equation*}
\hat{q}^k(s) = \left\{ 
- \sum_{m=1}^{\infty} \left(e^{-4amX} + e^{-4bmX}\right)
- \sum_{m=1}^{\infty} \sum_{n=1}^{\infty} (-1)^{n-1} \left(e^{-2(nb+ma)X} + e^{-2(mb+na)X} \right)
\right\}^k \hat{q}^0(s).
\end{equation*}
}
 On using binomial expansion, we get
  \begin{align*}
\hat{q}^k(s) &= (-1)^k e^{-4akX} \hat{q}^0(s) + (-1)^k e^{-4bkX} \hat{q}^0(s) \\
&\quad + \left( \sum_{l>k}^{\infty} D_l^{(k)} e^{-4blX} + \sum_{l>k}^{\infty} Z_l^{(k)} e^{-4alX} + \sum_{m+n \geq 2k}^{\infty} J_{m,n}^{(k)} e^{-2(nb + ma)X} \right) \hat{q}^0(s).
\end{align*}
Upon inversion of the Laplace transform, we obtain
 \begin{equation}
    q^k(t)=(-1)^k \mathcal{L}^{-1}\left\{ e^{-4akX}\hat q^0(s)\right \}+(-1)^k\mathcal{L}^{-1}\left\{ e^{-4bkX}\hat q^0(s)\right\}+\mathcal{L}^{-1}\left\{ \text{other terms}\right\}.
    \label{inverselaplace}
\end{equation}
By the \textbf{convolution theorem} (see \cite{schiff}), the inverse Laplace transform of a product is given by the convolution of their respective inverse transforms. Thus,
\[
\mathcal{L}^{-1}\left\{\widehat{\psi}(s) \cdot \widehat{\phi}(s)\right\} = (\psi * \phi)(t),
\]
where,  
\[
\widehat{\psi}(s) = e^{- \beta \sqrt{s^2 - \lambda e^{-\tau s}}}, \beta \geq0 \quad \text{and} \quad \widehat{\phi}(s) = \mathcal{L}\{\phi(t)\}.
\]
To compute the inverse Laplace transform of  
\[
\widehat{\psi}(s) = e^{- \beta \sqrt{s^2 - \lambda e^{-\tau s}}},
\]
we utilize \textbf{Efros' theorem}~\cite{sana2023dirichlet} which states that
if  \( \hat{F}(s) \) and \( \hat{P}(s) e^{-Q(s)\zeta } \) denote the Laplace transforms of given functions \( F(t) \) and \( P(t, \zeta ) \), with \( \zeta  \) treated as a parameter. Then,
\[
\mathcal{L}^{-1}\left\{\hat{F}(Q(s)) \cdot \hat{P}(s)\right\} = \int_0^\infty P(t,\zeta ) \cdot F(\zeta )\, d\zeta .
\]
For our application, we define: $\hat{F}(s) = e^{-\beta \sqrt{s}}$, $Q(s) = s^2 - \lambda e^{-\tau s}$ and $\hat{P}(s) = 1.$
Therefore, Efros’ theorem gives:
\[
\mathcal{L}^{-1}\left\{ e^{-\beta \sqrt{s^2 - \lambda e^{-\tau s}}} \right\}
= \int_0^\infty \mathcal{L}^{-1}\left\{ e^{-s^2 \zeta} \cdot e^{\lambda \zeta e^{-\tau s}} \right\} \cdot \mathcal{L}^{-1}\left\{ e^{-\beta \sqrt{s}} \right\} \, d\zeta.
\]
From standard Laplace transform tables~\cite{oberhettinger}:
\begin{align*}
\mathcal{L}^{-1}\left\{ e^{-\beta \sqrt{s}} \right\} &= \frac{\beta}{2\sqrt{\pi} t^{3/2}} \exp\left( - \frac{\beta^2}{4t} \right), \\
\mathcal{L}^{-1}\left\{ e^{-s^2 \zeta} \right\} &= \frac{1}{2 \sqrt{\pi \zeta}} \exp\left( - \frac{t^2}{4\zeta} \right).
\end{align*}
Using the expansion:
\[
e^{\lambda \zeta e^{-\tau s}} = \sum_{n=0}^{\infty} \frac{(\lambda \zeta)^n}{n!} e^{-n \tau s},
\]
we obtain:
\begin{align*}
&\int_0^\infty \mathcal{L}^{-1}\left\{ \sum_{n=0}^\infty \frac{(\lambda \zeta)^n}{n!} e^{-s^2 \zeta} e^{-n\tau s} \right\} \cdot \frac{\beta}{2\sqrt{\pi} \zeta^{3/2}} \exp\left( - \frac{\beta^2}{4\zeta} \right) \, d\zeta \\
&= \frac{\beta}{2\sqrt{\pi}} \int_0^\infty \mathcal{L}^{-1} \left\{ e^{-s^2 \zeta} \right\} \cdot \frac{1}{\zeta^{3/2}} \exp\left( - \frac{\beta^2}{4\zeta} \right) d\zeta 
+ \sum_{n=1}^{\infty} \mathcal{D}(\beta, \lambda, t - \beta - n\tau),
\end{align*}
The case $\lambda=0$, corresponding to no delay, is captured by the first term, representing the classical wave equation. For \( \lambda = 0 \), we retrieve:
\[
\mathcal{L}^{-1}\left\{ e^{-\beta s} \right\} \cdot \widehat{\phi}(s) = G(t - \beta) \cdot \phi(t - \beta),
\]
as a result of the second translation theorem (see \cite{oberhettinger}) where $G(t-\beta)$ is the Heaviside step function i.e. $G(t - \beta) =
\begin{cases}
0, & \text{for } t < \beta, \\
1, & \text{for } t \geq \beta.
\end{cases}$.\\
The remaining terms \( \mathcal{D}(\beta, \lambda, t - \beta - n\tau) \) encode the effects of successive time delays introduced by the exponential term \( e^{-n\tau s} \). These reflect a sequence of delayed wave responses (refer to Figure\ref{timeshift}).
Combining all contributions, we conclude:
\[
\mathcal{L}^{-1} \left\{ e^{- \beta \sqrt{s^2 - \lambda e^{-\tau s}}} \right\}
= \mathcal{D}(\beta, 0, t - \beta) + \sum_{n \in \mathbb{N}} \mathcal{D}(\beta, \lambda, t - \beta - n\tau).
\]
Accordingly, it follows from equation \eqref{inverselaplace} that
\begin{align*}
q^k(t)&=(-1)^k\left( d_1G(t-4ak/c)q^0(t-4ak/c)+d_2G(t-4bk/c)q^0(t-4bk/c)\right.\\
&\left.+d_3G(t-4ak/c-\tau)q^0(t-4ak/c-\tau)+d_4G(t-4bk/c-\tau)q^0(t-4bk/c-\tau)\right)\\
&+\sum_{\ell>k}\sum_{n\geq0}\left[z_{\ell,n}^{(k)}G(t-n\tau-4al/c)q^0(t-n\tau-4a\ell/c)\right.\\
&\left.+y_{\ell,n}^{(k)}G(t-n\tau-4b\ell/c)q^0(t-n\tau-4b\ell/c)\right]\\
&+\sum_{m+\nu\geq 2k}\sum_{n\geq0}j_{m,\nu,n}^{(k)}G(t-n\tau-2(am+b\nu)/c)q^0(t-n\tau-2(am+b\nu)/c).
\end{align*}
Given that \( d_i,\, z_{\ell,n}^{(k)},\, y_{\ell,n}^{(k)},\, j_{m,\nu,n}^{(k)} \in \mathbb{R} \). Suppose the time window \( T \) is chosen in order to satisfy the condition
\[
\frac{T}{k} \leq 4 \min\left\{ \frac{a}{c},\, \frac{b}{c} \right\},
\]
then, \( q^k(t) \equiv 0 \) $\forall$ \( t \in [0, T] \), ensuring that the NNWR algorithm attains the exact solution at the subsequent iteration. Thus, the NNWR algorithm for wave PDE with constant time delay achieves convergence in at most \( k+1 \) iterations, similar to the wave PDE without delay \cite{DD22}, as the delay terms only introduce additional time lag without affecting the convergence.
 \end{proof}

\begin{figure}[!b]
    \centering
    \subfloat[First subdomain is bigger]{\includegraphics[width=0.44\linewidth]{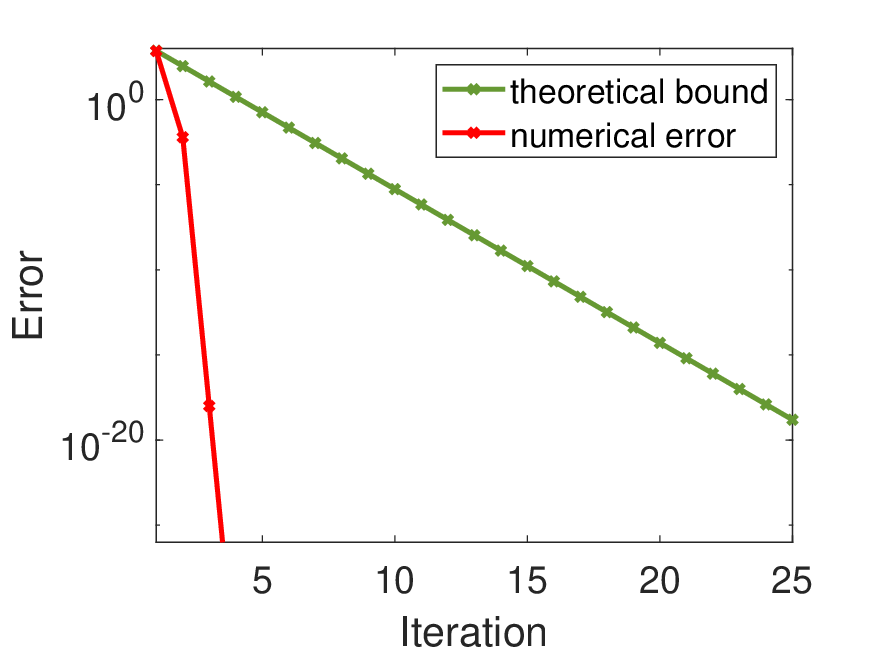}}
    \qquad
    \subfloat[Second subdomain is bigger]{\includegraphics[width=0.42\linewidth]{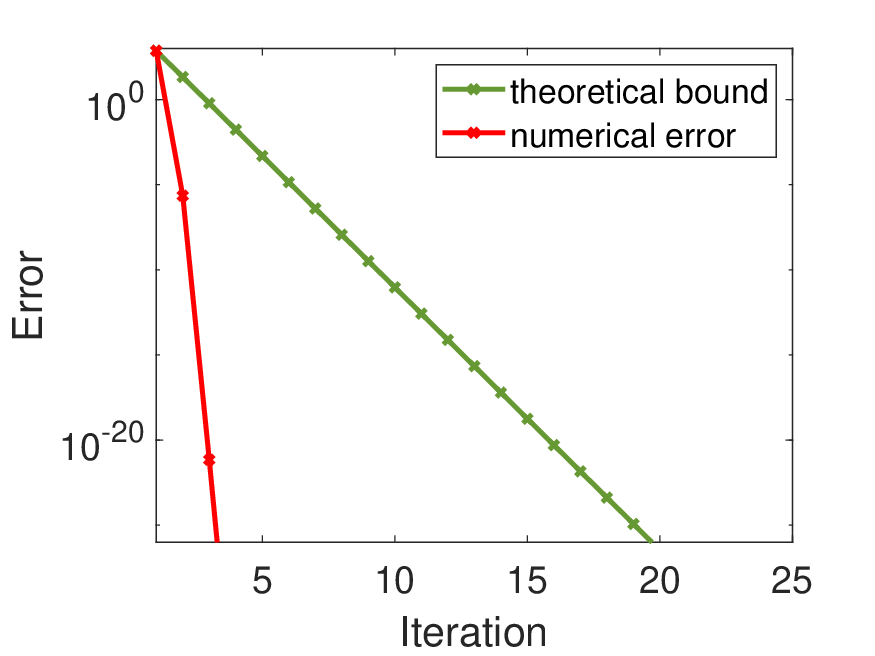}}
    \caption{Comparison of convergence of NNWR method numerically and their theoretical bounds for $\theta =1/4$, Left: when $D_1=(0,4),D_2=(4,6)$, Right: when $D_1=(0,2.5),D_2=(2.5,6)$.}
    \label{fourier_estimate}
\end{figure}

 \begin{figure}[!h]
    \centering
    \includegraphics[width=0.42 \linewidth]{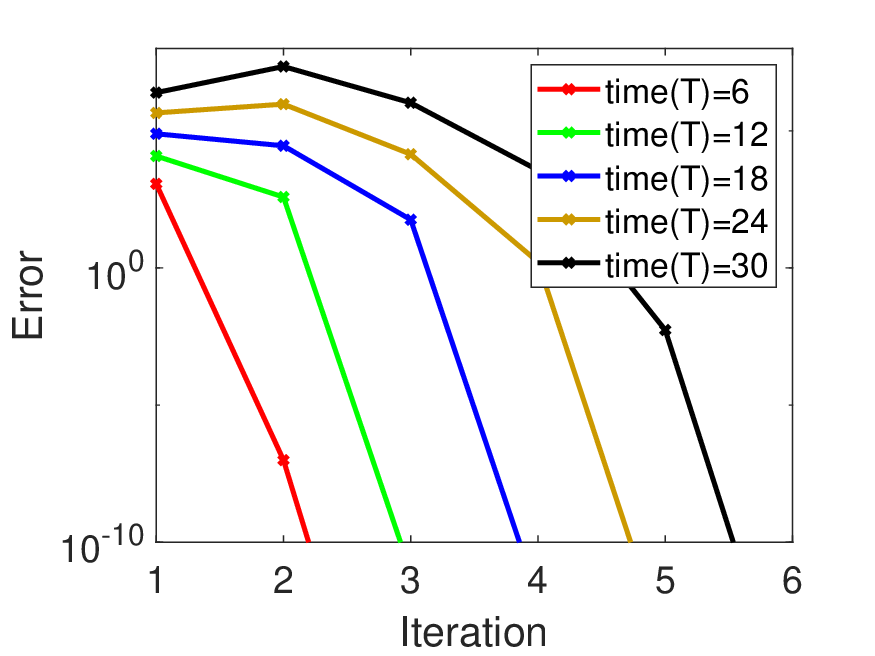}
    \caption{Convergence of Hyperbolic PDE with time delay for different time window T when $\theta =1/4$.} 
     \label{diftime_wave}
    \end{figure}

 \begin{figure}[!h]
    \centering
    \includegraphics[width=0.75 \linewidth]{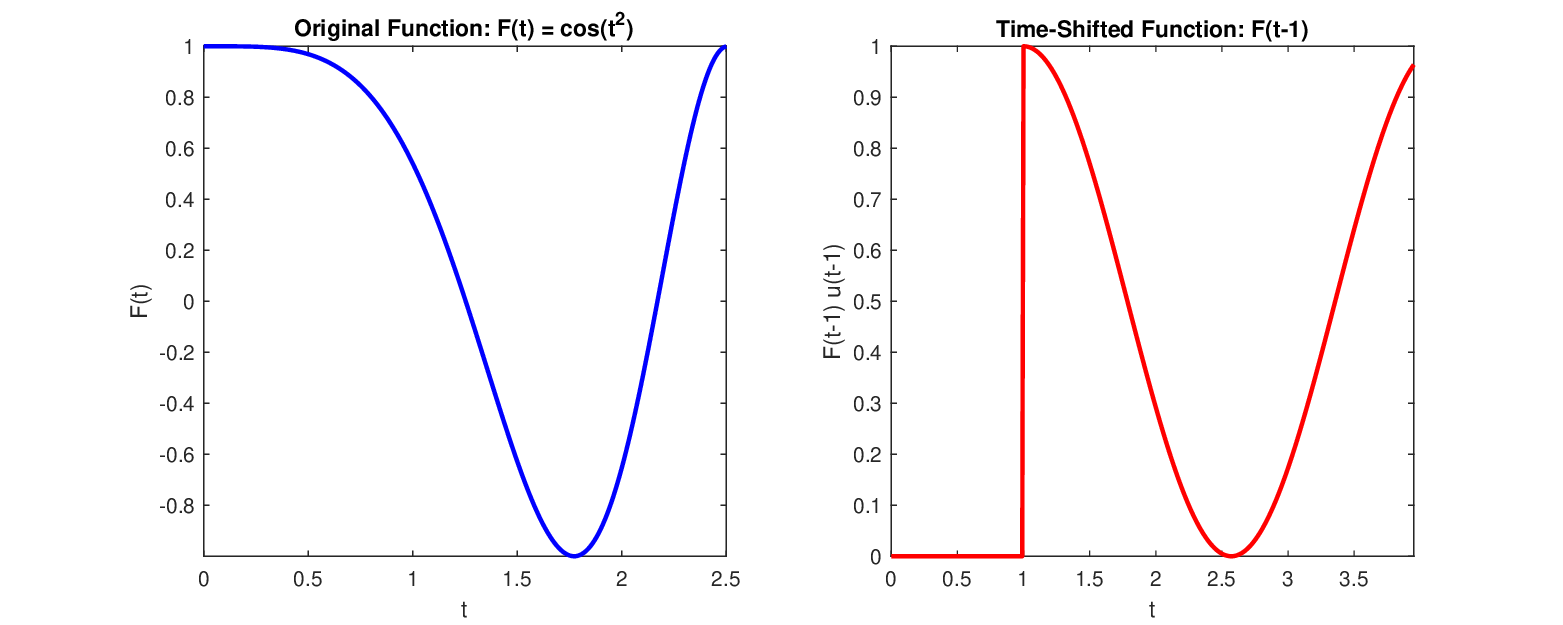}
    \caption{On the left, we show \( \mathcal{L}^{-1}\{ \hat{F}(s) \}\) which is $F(t)=\cos(t^2)$, and on the right, we illustrate the time-shifted version \( \mathcal{L}^{-1}\{ e^{-s} \hat{F}(s) \} \).
} 
     \label{timeshift}
    \end{figure}

\begin{figure}[!h]
    \centering
    \includegraphics[width=0.42 \linewidth]{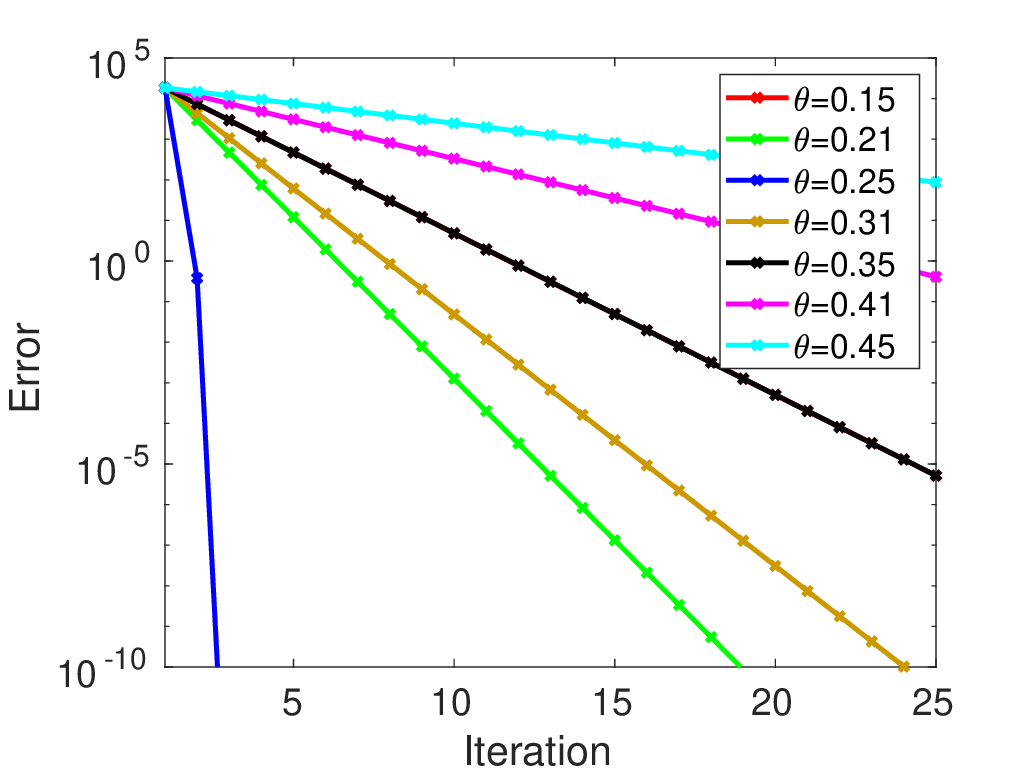}
    \caption{Convergence result of NNWR algorithm for Hyperbolic PDE in 2D with time delay for different values of $\theta$.} 
     \label{nnwr_2d}
    \end{figure}
\section{Numerical Illustration}
The computational domain is divided into non-uniform subdomains. For discretization of wave equation (\ref{eq_1}) we use the \textit{Leapfrog scheme}, employing second-order central difference in both space and time over a uniformly spaced grid points with time step size and spatial step size as \( \Delta t = 2 \times 10^{-2} \) and $\Delta x= 0.1$, respectively. Figure (\ref{fourier_estimate}) illustrates both the theoretical convergence bound and the observed numerical convergence pattern of the NNWR method when the optimal parameter choice \( \theta = 1/4 \). 

 Using the initial value \( q^0(t) = t^2 \) for \( t \in (0, T] \) on the interface, Figure~\ref{diftime_wave} shows the convergence trends of the NNWR method under varying time window lengths \( T \), again for the optimal parameter \( \theta = 1/4 \). Numerical results for the NNWR algorithm in two dimensions are presented for the spatial domain 
\(\Omega = (0, 8) \times (0, 8)\), which is further divided into two subdomains: 
\(D_1 = (0, 6) \times (0, 8)\) and \(D_2 = (6, 8) \times (0, 8)\) and the values of the discretization parameters are set as $\Delta x=\Delta t=\Delta y=0.1$. 
This represents an asymmetric decomposition of the 2D domain. For this configuration, the optimal value we get
for the relaxation parameter \(\theta\) is once again obtained as \(\frac{1}{4}\). See Figure~\ref{nnwr_2d} for illustration.

\section{Conclusion}
We have presented convergence estimates for the NNWR algorithm applied to hyperbolic PDEs with time delay, where the spatial region is decomposed in two asymmetric subdomains. Using the Fourier transform, we establish a linear convergence bound, while the Laplace transform analysis demonstrates finite-step convergence for the optimal value of relaxation parameter, which is $\theta=1/4$. These theoretical results are further supported by numerical experiments.

\printbibliography
\end{document}